\documentclass[11pt,reqno]{amsart}
\usepackage[margin=1.2in]{geometry}
\usepackage{mathrsfs}
\usepackage{amsthm}
\usepackage{amsmath}
\usepackage{amssymb}
\usepackage{amsfonts}
\usepackage[pdftex]{graphicx}  
\usepackage{latexsym}
\usepackage{comment}
\usepackage{ascmac}
\usepackage{caption}
\usepackage{mathtools}
\theoremstyle{plain}
\newtheorem*{syuA}{Main Theorem A}
\newtheorem*{syuB}{Main Theorem B}
\newtheorem{thm}{Theorem}[section]
\newtheorem{lem}[thm]{Lemma}
\newtheorem{cor}[thm]{Corollary}
\newtheorem{cla}[thm]{Claim}

\newtheorem{pro}[thm]{Proposition}
\theoremstyle{definition}
\newtheorem{df}[thm]{Definition}

\newtheorem{rem}[thm]{Remark}

\newtheorem*{prf*}{Proof}

\newtheorem*{pf*}{}
\newtheorem*{lem*}{LemmaA}
\newtheorem*{lm*}{LemmaB}
\newtheorem*{stra*}{Strategy for the proof of main result A}

\def\g2{l\ge2}

\def\la{\lambda}

\def\mdl{\mathcal{L}_d}
\def\m2l{\mathcal{L}_2}
\def\lpf{\underline{P}}
\makeatletter
\@namedef{subjclassname@2020}{%
  \textup{2020} Mathematics Subject Classification}
\makeatother
\title[Transversal family of Non-autonomous Conformal Iterated Function Systems]{Transversal family of non-autonomous conformal iterated function systems}

\author{Yuto Nakajima}
\address{Keio Institute of Pure and Applied Sciences (KiPAS)\\ Department of Mathematics,
Keio University, Yokohama,
223-8522, JAPAN}
\email{nakajimayuto@math.keio.ac.jp}

\date{}

\begin{document}
\maketitle
\begin{abstract}
We study Non-autonomous Iterated Function Systems (NIFSs) with overlaps. A NIFS on a compact subset $X\subset\mathbb{R}^m$ is a sequence $\Phi=(\{\phi^{(j)}_{i}\}_{i\in I^{(j)}})_{j=1}^{\infty}$ of collections of uniformly contracting maps $\phi^{(j)}_{i}: X\rightarrow X$, where $I^{(j)}$ is a finite set. In comparison to usual iterated function systems, we allow the contractions $\phi^{(j)}_{i}$ applied at each step $j$ to depend on $j$. In this paper, we focus on a family of parameterized NIFSs on $\mathbb{R}^m$. Here, we do not assume the open set condition. We show that if a $d-$parameter family of such systems satisfies the transversality condition, then for almost every parameter value the Hausdorff dimension of the limit set is the minimum of $m$ and the Bowen dimension. Moreover, we give an example of a family $\{\Phi_t\}_{t\in U}$ of parameterized NIFSs such that $\{\Phi_t\}_{t\in U}$ satisfies the transversality condition but $\Phi_t$ does not satisfy the open set condition for any $t\in U$. 


\end{abstract}
\tableofcontents
\section{Introduction}

The aim of this paper is to develop the dimension theory of Iterated Function Systems (IFSs). 
An IFS on a compact subset $X\subset\mathbb{R}^m$ is a collection $\{\phi_1,...,\phi_n\}$  of uniformly contracting maps $\phi_i : X \rightarrow X$. It is well-known that there uniquely exists a non-empty compact subset $A\subset X$ such that $$A=\bigcup_{i=1}^n \phi_i(A),$$ called the limit set of the IFS (\cite{Hut}). 
 In order to analyze the fine-scale structure of the limit set, it is important to estimate the dimension of the limit set. If the conformal IFS satisfies some separating condition, the Hausdorff dimension of the limit set is the zero of the pressure function corresponding to the IFS (see e.g., \cite {MoU, Fal1}). 
 
It is natural to consider a non-autonomous version of the IFS as an application for various problems (see e.g., \cite{WuX17, MaUr22, N}). A Non-autonomous Iterated Function System (NIFS) on a compact subset $X\subset\mathbb{R}^m$ is a sequence $\Phi=(\{\phi^{(j)}_{i}\}_{i\in I^{(j)}})_{j=1}^{\infty}$ of collections of uniformly contracting maps $\phi^{(j)}_{i}: X\rightarrow X$, where $I^{(j)}$ is a finite set. The system $\Phi$ is called an IFS if the collections $\{\phi^{(j)}_{i}\}_{i\in I^{(j)}}$ are independent of $j$. In comparison to usual IFSs, we allow the contractions $\phi^{(j)}_{i}$ applied at each step $j$ to depend on $j$.  
As a remarkable result for NIFSs, we mention the theory of Non-autonomous Conformal Iterated Function Systems (NCIFSs) which is introduced by Rempe-Gillen and Urba$\acute{\rm n}$ski \cite{RU}. A NCIFS on a compact subset $X\subset\mathbb{R}^m$ is a NIFS which consists of collections of conformal maps $\phi^{(j)}_{i} : X \rightarrow X$ satisfying some mild conditions containing the Open Set Condition (OSC) which is defined as follows. 
 We say that a NIFS $(\{\phi^{(j)}_{i}\}_{i\in I^{(j)}})_{j=1}^{\infty}$ on a compact subset $X$ with ${\rm int}(X)\neq \emptyset$ satisfies the OSC if for all $j\in\mathbb{N}$ and all distinct indices $a,b\in I^{(j)}$, 

\begin{align}
\label{osc}
\phi_a^{(j)}({\rm int}(X))\cap \phi_b^{(j)}({\rm int}(X))=\emptyset.
\end{align}  
Then the limit set of the NCIFS $\Phi=(\{\phi^{(j)}_{i}\}_{i\in I^{(j)}})_{j=1}^{\infty}$ is defined as the set of possible limit points of sequences $\phi^{(1)}_{\omega_1}(\phi^{(2)}_{\omega_2}...(\phi^{(i)}_{\omega_i}(x))...)),$ $\omega_j\in I^{(j)}$ for all $j\in\{1,2,..., i\}$, $x\in X$. The condition (\ref{osc}) imposes restrictions on the overlaps in the limit set of the NCIFS. Moreover, Rempe-Gillen and Urba$\acute{\rm n}$ski introduced {\em the lower pressure function} $\lpf_{\Phi}: [0, \infty)\rightarrow [-\infty, \infty]$ of the NCIFS $\Phi$. Then the {\em Bowen dimension} $s_{\Phi}$ of the NCIFS $\Phi$ is defined by $s_{\Phi}={\rm sup}\{s\ge0\ :\ \underline{P}_{\Phi}(s)>0\}={\rm inf}\{s\ge0\ :\ \underline{P}_{\Phi}(s)<0\}.$
Rempe-Gillen and Urba$\acute{\rm n}$ski proved that the Hausdorff dimension of the limit set is the Bowen dimension of the NCIFS (\cite[1.1 Theorem]{RU}).  For related results for non-autonomous systems satisfying some separating condition, see \cite{HZ, N}.

In this paper, we consider NIFSs in the complicated overlapping case. Here, we do not assume the OSC (\ref{osc}). 
To investigate NIFSs with overlaps, we focus on a family of parameterized NIFSs on $\mathbb{R}^m$ by using the transversality method. The transversality method is utilized for the dimension estimation of the limit sets of parameterized IFSs involving some complicated overlaps (see \cite{PoS, S2, SSU, J, JP}). This method also provides a crucial tool for the absolute continuity of the invariant measures (e.g., Bernoulli convolutions) of IFSs or some variations (see \cite{S1, PS2, SSU2, T}). For some general family of functions with the transversality condition, see  \cite{SSU, MU, SU}. One of the aims in this paper is to provide a non-autonomous version of the classical transversality method (Main Theorem A). Furthermore, we give a family of NIFSs for which we essentially use Main Theorem A to estimate the Hausdorff dimension of limit sets of the NIFSs (Main Theorem B). 

The paper is organized as follows. In section 2 we introduce transversal families of non-autonomous conformal iterated function systems on $\mathbb{R}^m$.  In section 4 we consider some conditions for a deeper understanding of the system given in section 2. Section 4 is devoted to a proof of one of the main results. As preliminaries for the proof, we give some lemma for conformal maps on $\mathbb{R}^m$ and construct a Gibbs-like measure on the symbolic space. Finally, we give the proof by using the transversality method. In section 5 we give an example of a family $\{\Phi_t\}_{t\in U}$ of parameterized NIFSs such that $\{\Phi_t\}_{t\in U}$ satisfies the transversality condition but $\Phi_t$ does not satisfy the open set condition (\ref{osc}) for any $t\in U$.

\section{Main results}
In this section we present a framework of transversal families of non-autonomous conformal iterated function systems and give the main results on them. For each $j\in \mathbb{N},$ let $I^{(j)}$ be a finite set. For any $n, k\in \mathbb{N}$ with $n\le k,$
we introduce index sets
\begin{align*}
&I^k_n:=\prod_{j=n}^k I^{(j)}, I^{\infty}_n:=\prod_{j=n}^{\infty}I^{(j)}, I^n:=\prod_{j=1}^n I^{(j)}, \mbox{and}\ I^{\infty}:=\prod_{j=1}^{\infty}I^{(j)}.
\end{align*}

Let $U\subset\mathbb{R}^d$. For any $t\in U$, let $\Phi_t=(\Phi^{(j)}_t)_{j=1}^{\infty}$ be a sequence of collections of maps on a set $X\subset \mathbb{R}^m$, where $$\Phi^{(j)}_t=\{\phi_{i, t}^{(j)}:X\rightarrow X\}_{i\in I^{(j)}}.$$ 
Let $n, k\in \mathbb{N}$ with $n\le k$. For any $\omega=\omega_n\omega_{n+1}\cdots\omega_k\in I^k_n,$ we set $$\phi_{\omega, t}:=\phi_{\omega_n, t}^{(n)}\circ\cdots\circ \phi_{\omega_{k}, t}^{(k)}.$$
Let $n\in \mathbb{N}$. For any $\omega=\omega_{n}\omega_{n+1}\cdots\in I^{\infty}_n$ and any $j\in \mathbb{N}$, we set $$\omega|_j:=\omega_{n}\omega_{n+1}\cdots\omega_{n+j-1}\in I^{n+j-1}_n.$$ 
Let $V\subset \mathbb{R}^m$ be an open set and let $\phi: V\rightarrow \phi(V)$ be a diffeomorphism. We denote by $D\phi(x)$ the derivative of $\phi$ evaluated at $x$. We say that $\phi$ is {\em conformal} if for any $x\in V$ $D\phi(x):\mathbb{R}^m\rightarrow \mathbb{R}^m$ is a similarity linear map, that is, $D\phi(x)=c_x\cdot A_x$, where $c_x>0$ and $A_x$ is an orthogonal matrix. For any conformal map $\phi: V\rightarrow \phi(V),$ we denote by $|D\phi(x)|$ its scaling factor at $x$, that is, if we set $D\phi(x)=c_x\cdot A_x$ we have $|D\phi(x)|=c_x$. For any set $A\subset V$, we set 
$$||D\phi||_A:={\rm sup}\{|D\phi(x)|\ :\ x\in A\}.$$
We denote by $\mathcal{L}_d$ the $d$-dimensional Lebesgue measure on $\mathbb{R}^d$. 
We introduce {\em the transversal family of non-autonomous conformal iterated function systems} by employing the settings in \cite{RU} and \cite{SSU}.
\begin{df}[Transversal family of non-autonomous conformal iterated function systems]
\label{tncifs}
Let $m\in \mathbb{N}$ and let $X\subset \mathbb{R}^m$ be a non-empty compact convex set. 
Let  $d\in \mathbb{N}$ and let $U\subset \mathbb{R}^d$ be a non-empty open set. For each $j\in \mathbb{N}$, let $I^{(j)}$ be a finite set. Let $t\in U.$ For any $j\in \mathbb{N}$, let $\Phi^{(j)}_t$ be a collection $\{\phi_{i, t}^{(j)}:X\rightarrow X\}_{i\in I^{(j)}}$ of maps $\phi_{i, t}^{(j)}$ on $X$. Let $\Phi_t=(\Phi^{(j)}_t)_{j=1}^{\infty}$. We say that $\{\Phi_t\}_{t\in U}$ is a Transversal family of Non-autonomous Conformal Iterated Function Systems (TNCIFS) if $\{\Phi_t\}_{t\in U}$ satisfies the following six conditions.
\begin{enumerate}
\item[1.] {\em Conformality} : There exists an open connected set $V\supset X$(independent of $i, j$ and $t$) such that for any $i, j$ and $t\in U$, $\phi_{i, t}^{(j)}$ extends to a $C^{1}$ conformal map on $V$ such that $\phi_{i, t}^{(j)}(V)\subset V$.
\item[2.] {\em Uniform contraction} : There is a constant $0< \gamma <1$ such that for any $t\in U,$ any $n\in \mathbb{N}$, any $\omega\in I_n^{\infty}$ and any  $j\in \mathbb{N}$, 
\begin{align*}
|D\phi_{\omega|_j, t}(x)|\le\gamma^j
\end{align*}
for any $x\in V\supset X.$
\item[3.] {\em Bounded distortion} : There exists a continuous 
function $K: U\rightarrow [1, \infty)$ such that for any $t\in U,$ any $n\in \mathbb{N}$, any $\omega\in I_n^{\infty}$ and any $j\in \mathbb{N},$

\begin{align}
\label{bd}
|D\phi_{\omega|_j, t}(x_1)|\le K(t)|D\phi_{\omega|_j, t}(x_2)|
\end{align}
for any $x_1,x_2\in V$. 
\item[4.] {\em Distortion continuity} : For any $\eta>0$ and $t_0\in U$, there exists $\delta=\delta(\eta, t_0)>0$ such that for any $t\in U$ with $|t-t_0|\le \delta$, for any $n, j\in \mathbb{N},$ and for any $\omega\in I_n^{\infty},$
\begin{align} 
\label{dc} 
\exp(-j \eta)\le \frac{||D\phi_{\omega|_j, t_0}||_X}{||D\phi_{\omega|_j, t}||_X}\le \exp(j \eta).
\end{align}

We now define the {\em address map} as follows. Let $t\in U$. For all $n\in \mathbb{N}$ and all $\omega\in I^{\infty}_n$, $$\bigcap_{j=1}^{\infty}\phi_{\omega|_{j}, t}(X)$$ is a singleton by the uniform contraction property. It is denoted by $\{x_{{\omega}, n, t}\}$. The map $$\pi_{n, t}\colon I^{\infty}_n \rightarrow X$$ is defined by $\omega \mapsto x_{{\omega}, n, t}$. Then $\pi_{n, t}$ is called {\em the $n$-th address map corresponding to $t$}. 
Note that for any $t\in U$ and $n\in \mathbb{N}$ the map $\pi_{n, t}$ is continuous with respect to the product topology on $I^{\infty}_n$.
\item[5.] {\em Continuity} : Let $n\in \mathbb{N}$. The function $U \times I^{\infty}_n\ni(t, \omega)\mapsto \pi_{n, t}(\omega)$ is continuous.

\item[6.] {\em Transversality condition} : For any compact subset $G\subset U$ there exists a sequence $\{C_n\}_{n=1}^{\infty}$ of positive constants with
$$\lim_{n\to \infty}\frac{\log C_n}{n}=0$$
 such that for all $\omega,\tau\in I_n^{\infty}$ with $\omega_n\neq \tau_n$ and for all $r>0$,   

\begin{align*}
\mathcal{L}_d\left(\{t\in G\ :\ |\pi_{n, t}(\omega)-\pi_{n, t}(\tau)|\le r\}\right)\le C_n r^m.
\end{align*}
\end{enumerate}
\end{df}
\begin{rem}
\label{rem1}
\begin{enumerate}
\item[(i)] If $m\ge 2,$ the conformality condition implies the bounded distortion condition. For the details, see \cite[page. 1984 Remark]{RU}. 
\item[(ii)]Let $n\in \mathbb{N}$ and let $t\in U$. Then for any $\omega\in I_{n}^{\infty},$ we have
$$\pi_{n, t}(\omega)=\lim_{j\to \infty} \phi_{{\omega|_j}, t}(x),$$
where $x\in X.$
\begin{proof}
Since $\phi_{{\omega|_{j+k}}, t}(x)=\phi_{{\omega|_{j}}, t}(\phi_{{\omega_{j+1}\cdots \omega_{j+k}}, t}(x))\in \phi_{\omega|_{j}, t}(X)$ for any $j, k\in \mathbb{N},$ we have $\lim_{j\to \infty} \phi_{{\omega|_j}, t}(x)=\lim_{k\to \infty} \phi_{{\omega|_{j+k}}, t}(x)\in \phi_{\omega|_{j}, t}(X)$ for any $j\in \mathbb{N}.$ Hence $\lim_{j\to \infty} \phi_{{\omega|_j}, t}(x)=\pi_{n, t}(\omega).$
\end{proof}

\item[(iii)]In the case of usual IFSs, the constants $C_n$ in the transversality condition are independent of $n$ since the $n$-th address maps $\pi_{n, t}$ are independent of $n$. 
\end{enumerate}
\end{rem}
In Section 3, we give some further conditions under which the conditions 3, 4, 5 in Definition \ref{tncifs} hold. 

Let $\{\Phi_t\}_{t\in U}$ be a TNCIFS. 
For any $t\in U$, we define the lower pressure function $\underline{P}_t: [0, \infty)\rightarrow [-\infty, \infty]$ of $\Phi_t$ as follows. For any $s\ge 0$ and $n\in \mathbb{N}$, we set
\begin{align}
\label{PF1}
Z_{n, t}(s):=\sum_{\omega\in I^n}(||D\phi_{\omega, t}||_X)^s,
\end{align}
and 
\begin{align}
\label{PF2}
\underline{P}_t(s):=\liminf\limits_{n\rightarrow \infty}\frac{1}{n}\log Z_{n, t}(s)\in[-\infty,\infty].
\end{align}
By \cite[Lemma 2.6]{RU}, the lower pressure function has the following monotonicity. If $s_1<s_2$, then either both $\underline{P}_t(s_1)$ and $\underline{P}_t(s_2)$ are equal to $\infty$, both are equal to $-\infty$, or $\underline{P}_t(s_1)>\underline{P}_t(s_2)$. Then for any $t\in U$, we set
\begin{align}
\label{PF3}
s(t)=s(\Phi_t):={\rm sup}\{s\ge0\ :\ \underline{P}_t(s)>0\}={\rm inf}\{s\ge0\ :\ \underline{P}_t(s)<0\},
\end{align}
where we set ${\rm sup}\ \emptyset=0$ and ${\rm inf}\ \emptyset=\infty.$ The value $s(t)$ is called the Bowen dimension of $\Phi_t$. 
We define the limit set $J_t$ of $\Phi_t$ by 
\begin{align*}
J_{t}:=\pi_{1, t}(I^{\infty}).
\end{align*}
\begin{rem}
The Hausdorff dimension function $t\mapsto \dim_H(J_t)$ is Borel measurable by the continuity condition. To see this, let $t_0\in U, \omega\in I^{\infty}.$ By the continuity condition, for any $\epsilon>0,$ there exists a neighborhood $N_{t_0, \omega}\subset U$ of $t_0$ and a neighborhood $M_{t_0, \omega}\subset I^{\infty}$ of $\omega$ such that for any $(t, \tau)\in N_{t_0, \omega}\times M_{t_0, \omega},$ $|\pi_{1, t}(\tau)-\pi_{1, t_0}(\omega)|<\epsilon.$ Since $I^{\infty}$ is a compact set, we can find $\omega_1,..., \omega_k\in I^{\infty}$ such that $I^{\infty}=M_{t_0, \omega_1}\cup\cdots \cup M_{t_0, \omega_k}.$ Then for any $t\in U$ with $t\in N_{t_0, \omega_1}\cap\cdots \cap N_{t_0, \omega_k},$ any $i\in \{1,..., k\},$ and any $\tau\in M_{t_0, \omega_i},$ $|\pi_{1, t}(\tau)-\pi_{1, t_0}(\omega_i)|<\epsilon,$ which implies that the map $t\mapsto J_t$ is continuous with respect to the Hausdorff metric. Furthermore, the Hausdorff dimension function on the space of compact subsets of $\mathbb{R}^m$ endowed with the Hausdorff metric is Borel measurable (see \cite[Theorem 2.1]{MM}). Hence we obtain the desired property.
\end{rem}
For $A\subset {\mathbb R}^m$ we denote by $\dim_H(A)$ the Hausdorff dimension of $A$. We now present one of the main results of this paper. 
\begin{syuA}[Theorem \ref{thm1}]
Let $\{\Phi_t\}_{t\in U}$ be a TNCIFS. 
Then 
\begin{enumerate}
\item[(i)] \[\dim_H(J_t)= \min\{m, s(t)\}\ \text{for}\ \mdl-\text{a.e.}\ t\in U;\]

\item[(ii)] \[\mathcal{L}_m(J_t)>0\ \text{for}\ \mdl-\text{a.e.}\ t\in \{t\in U\ : s(t)>m\}.\]
\end{enumerate}
\end{syuA}
Main Theorem A is a generalization of \cite[Theorem 3.1]{SSU}. We illustrate Main Theorem A by presenting the following important example. We set $$X=\left\{z\in \mathbb{C}\ :\ |z|\le \frac{1}{1-2\times 5^{-5/8}}\right\},U=\left\{t\in \mathbb{C}\ :\ |t|< 2\times 5^{-5/8},\ t\notin \mathbb{R}\right\}.$$ Note that $2\times 5^{-5/8}\approx 0.73143>1/\sqrt{2}.$ Let $t\in U$. For each $j\in \mathbb{N},$ we define the maps $\phi^{(j)}_{1, t}: X\rightarrow X$ and $\phi^{(j)}_{2, t}: X\rightarrow X$ by \[\phi^{(j)}_{1, t}(z)=t z\ \text{and}\ \phi^{(j)}_{2, t}(z)= tz+\frac{1}{j}\] respectively. For each $j\in \mathbb{N},$ we set $$\Phi^{(j)}_t=\{\phi^{(j)}_{1, t}, \phi^{(j)}_{2, t}\}=\left\{z \mapsto t z, z\mapsto tz+\frac{1}{j}\right\}$$ and $\Phi_t=(\Phi^{(j)}_t)_{j=1}^{\infty}.$ We now present the following theorem, which is the second main result of this paper.
\begin{syuB}[Proposition \ref{notosc} and Proposition \ref{extncifs}]
The family $\{\Phi_t\}_{t\in U}$ of parameterized systems is a TNCIFS but $\Phi_t$ does not satisfy the open set condition $(\ref{osc})$ for any $t\in U$. 
\end{syuB}
Note that we cannot apply the framework of Rempe-Gillen and Urba$\acute{\rm n}$ski \cite{RU} to the study of the limit set $J_t$ of $\Phi_t$ since $\Phi_t$ does not satisfy the open set condition (\ref{osc}) for any $t\in U$. We calculate the lower pressure function $\lpf_{t}$ of $\Phi_t$ as follows. For any $s\in [0, \infty),$

\begin{align*}
\lpf_{t}(s)&=\liminf\limits_{n\rightarrow \infty}\frac{1}{n}\log \sum_{\omega\in I^n}(||D\phi_{\omega, t}||_X)^s
=\liminf\limits_{n\rightarrow \infty}\frac{1}{n}\log \sum_{\omega\in I^n}|t|^{ns}\\
&=\liminf\limits_{n\rightarrow \infty}\frac{1}{n}\log (2^n|t|^{ns})
=\log 2+s\log |t|.
\end{align*}

Hence for each $t\in U,$ $\lpf_{t}(s)$ has the zero $$s(t)=\frac{\log2}{-\log |t|}.$$ By Main Theorem A, we obtain the following.
\begin{cor} Let $J_t$ be the limit set corresponding to $t$. Then
 $$\dim_H(J_t)=\min\{2, s(t)\}=s(t)$$ for a.e. $t\in \{t\in \mathbb{C}\ :\ |t|\le1/\sqrt{2}, t\notin \mathbb{R}\}(\subset U)$ and $$\mathcal{L}_2(J_t)>0$$ for a.e. $t\in \{t\in \mathbb{C}\ :\ 1/\sqrt{2}<|t|<2\times 5^{-5/8}, t\notin \mathbb{R}\}(\subset U).$
\end{cor}

\section{On the conditions of transversal families of non-autonomous conformal iterated function systems}
In this section, we give some further conditions under which the conditions 3, 4, 5 in Definition \ref{tncifs} hold. In the rest of this section, suppose that $\{\Phi_t\}_{t\in U}=\{(\{\phi_{i, t}^{(j)}:X\rightarrow X\}_{i\in I^{(j)}})_{j=1}^{\infty}\}_{t\in U}$ is a family of parameterized NIFSs with the conformality condition and the uniform contraction condition. 
\begin{pro}
If $m=1,$ suppose that there exist positive real valued continuous functions $\alpha, \beta$, and $C$ on $U$ such that for any $t\in U$, any $j\in \mathbb{N}$, any $i\in I^{(j)},$ \[\left|D\phi^{(j)}_{i, t}(x)-D\phi^{(j)}_{i, t}(y)\right|\le C(t)|x-y|^{\alpha(t)}\] and \[\beta(t)<|D\phi^{(j)}_{i, t}(x)|\] for any points $x, y$ in a bounded open interval $V\supset X.$ Then the bounded distortion condition holds.
\end{pro}
\begin{proof}
Let $t\in U$, let $n, j\in \mathbb{N}$ and let $\omega\in I^{n+j-1}_n.$ Then for any $x, y\in V,$ 
\begin{align*}
\left|\log\frac{|D\phi_{\omega, t}(x)|}{|D\phi_{\omega, t}(y)|}\right|&=\left|\sum_{k=n}^{n+j-1}\log\left|D\phi^{(k)}_{\omega_k, t}(\phi_{{\omega_{k+1}\cdots \omega_{n+j-1}}, t}(x))\right|-\log\left|D\phi^{(k)}_{\omega_k, t}(\phi_{{\omega_{k+1}\cdots \omega_{n+j-1}}, t}(y))\right|\right|\\
&\le \sum_{k=n}^{n+j-1}\left|\log\left|D\phi^{(k)}_{\omega_k, t}(\phi_{{\omega_{k+1}\cdots \omega_{n+j-1}}, t}(x))\right|-\log\left|D\phi^{(k)}_{\omega_k, t}(\phi_{{\omega_{k+1}\cdots \omega_{n+j-1}}, t}(y))\right|\right|\\
&\le \beta(t)^{-1}\sum_{k=n}^{n+j-1}\left|\left|D\phi^{(k)}_{\omega_k, t}(\phi_{{\omega_{k+1}\cdots \omega_{n+j-1}}, t}(x))\right|-\left|D\phi^{(k)}_{\omega_k, t}(\phi_{{\omega_{k+1}\cdots \omega_{n+j-1}}, t}(y))\right|\right|\\
&\le \beta(t)^{-1}\sum_{k=n}^{n+j-1}C(t)\left|\phi_{{\omega_{k+1}\cdots \omega_{n+j-1}}, t}(x)-\phi_{{\omega_{k+1}\cdots \omega_{n+j-1}}, t}(y)\right|^{\alpha(t)}\\
&\le \beta(t)^{-1}C(t)\sum_{k=n}^{n+j-1}\gamma^{\alpha(t)(n+j-k-1)}|x-y|^{\alpha(t)}\le \frac{C(t)}{\beta(t)\left(1-\gamma^{\alpha(t)}\right)}\sup_{x, y\in V}|x-y|^{\alpha(t)},
\end{align*}
where $\gamma$ is the constant coming from the uniform contraction condition.
\end{proof}
Let $t_1, t_2\in U.$ For any $j\in \mathbb{N}$ and any $i\in I^{(j)},$ we set $$||\phi_{i, t_1}^{(j)}-\phi_{i, t_2}^{(j)}||_{\infty}=\sup\{|\phi_{i, t_1}^{(j)}(x)-\phi_{i, t_2}^{(j)}(x)|\ :\ x\in X\}$$ and $$||\Phi_{t_1}-\Phi_{t_2}||_{\infty}=\sup_{j\in \mathbb{N}}\max_{i\in I^{(j)}}||\phi_{i, t_1}^{(j)}-\phi_{i, t_2}^{(j)}||_{\infty}.$$
\begin{pro}
\label{condtncifs2}
 Suppose that $\{\Phi_t\}_{t\in U}$ satisfies the following:
\begin{enumerate}
\item[(a)] $\text{For any}\ \epsilon>0\ \text{and}\ t_0\in U,\ \text{there exists}\ \delta>0\ \text{such that for any}$ $t\in U$ with $|t_0-t|<\delta,$ \[||\Phi_{t_0}-\Phi_{t}||_{\infty}<\epsilon.\]
\end{enumerate} Then the continuity condition holds.
\end{pro}
\begin{proof}
Let $n\in \mathbb{N}.$ Fix $\epsilon>0$ and $(t_0, \omega_0)\in U\times I_n^{\infty}.$ Since the map $\pi_{n, t_0}$ is continuous at $\omega_0\in I_n^{\infty},$ there exists a neighborhood $N_{t_0, \omega_0}\subset I_n^{\infty}$ of $\omega_0$ such that for any $\tau\in N_{t_0, \omega_0}$ 
\begin{align}
\label{rempi1}
|\pi_{n, t_0}(\omega_0)-\pi_{n, t_0}(\tau)|<\frac{\epsilon}{2}.
\end{align}
Furthermore, we prove the following claim.
\begin{cla}
\label{CLAIM}
For any $t\in U,$ any $j\in \mathbb{N},$ any $\tau=\tau_n\tau_{n+1}\cdots\in I_n^{\infty},$ and any point $x\in X,$ \begin{align}\label{CLAIM1}|\phi_{\tau|_j, t_0}(x)-\phi_{\tau|_j, t}(x)|\le \frac{1}{1-\gamma}||\Phi_{t_0}-\Phi_{t}||_{\infty},\end{align} where $\gamma$ is the constant coming from the uniform contraction condition.
\end{cla}
\begin{proof}[Proof of Claim $\ref{CLAIM}$]
\begin{align*}
|\phi_{\tau|_j, t_0}(x)-\phi_{\tau|_j, t}(x)|\le &|\phi_{\tau|_{j-1}, t_0}(\phi_{\tau_{n+j-1}, t_0}(x))-\phi_{\tau|_{j-1}, t_0}(\phi_{\tau_{n+j-1}, t}(x))|\\&+|\phi_{\tau|_{j-1}, t_0}(\phi_{\tau_{n+j-1}, t}(x))-\phi_{\tau|_{j-1}, t}(\phi_{\tau_{n+j-1}, t}(x))|\\&\le \gamma^{j-1} ||\Phi_{t_0}-\Phi_{t}||_{\infty}+|\phi_{\tau|_{j-1}, t_0}(x^{\prime})-\phi_{\tau|_{j-1}, t}(x^{\prime})|\\&\hspace{-40pt}(\mbox{by the uniform contraction condition and setting}\ x^{\prime}=\phi_{\tau_{n+j-1}, t}(x))\\&\le \gamma^{j-1} ||\Phi_{t_0}-\Phi_{t}||_{\infty}+\gamma^{j-2} ||\Phi_{t_0}-\Phi_{t}||_{\infty}+\cdots +||\Phi_{t_0}-\Phi_{t}||_{\infty}\\&\le \frac{1}{1-\gamma}||\Phi_{t_0}-\Phi_{t}||_{\infty}.
\end{align*}
\end{proof}

Letting $j\to \infty$ in (\ref{CLAIM1}), by Remark \ref{rem1} (ii) we have for any $\tau\in I_n^{\infty},$ $|\pi_{n, t_0}(\tau)-\pi_{n, t}(\tau)|\le {1}/{(1-\gamma)}||\Phi_{t_0}-\Phi_{t}||_{\infty}.$
By (a), there exists a neighborhood $M_{t_0}\subset U$ of $t_0$ such that for any $t\in M_{t_0}$ and any $\tau\in I_n^{\infty},$ 
\begin{align}
\label{rempi2}
|\pi_{n, t_0}(\tau)-\pi_{n, t}(\tau)|\le \frac{1}{1-\gamma}\cdot \frac{(1-\gamma)\epsilon}{2}=\frac{\epsilon}{2}.
\end{align}
By (\ref{rempi1}) and (\ref{rempi2}) we have for any $(t, \tau)\in M_{t_0}\times N_{t_0, \omega_0}\subset U\times I_n^{\infty}$, $|\pi_{n, t_0}(\omega_0)-\pi_{n, t}(\tau)|\le |\pi_{n, t_0}(\omega_0)-\pi_{n, t_0}(\tau)|+|\pi_{n, t_0}(\tau)-\pi_{n, t}(\tau)|\le {\epsilon}/{2}+{\epsilon}/{2}=\epsilon.$
\end{proof}
We show that the conditions introduced in \cite[Section 10]{RU} imply the distortion continuity condition and the continuity condition. Let $t_1, t_2\in U.$ For any $j\in \mathbb{N}$ and any $i\in I^{(j)},$ we set $$||D\phi_{i, t_1}^{(j)}-D\phi_{i, t_2}^{(j)}||_{\infty}=\sup\{|D\phi_{i, t_1}^{(j)}(x)-D\phi_{i, t_2}^{(j)}(x)|\ :\ x\in X\}$$ and $$||D\Phi_{t_1}-D\Phi_{t_2}||_{\infty}=\sup_{j\in \mathbb{N}}\max_{i\in I^{(j)}}||D\phi_{i, t_1}^{(j)}-D\phi_{i, t_2}^{(j)}||_{\infty}.$$ 
\begin{pro}
Suppose that $\{\Phi_t\}_{t\in U}$ satisfies the following:
\begin{enumerate}
\item[(b)] $\text{For any}\ \epsilon>0\ \text{and}\ t_0\in U,\ \text{there exists}\ \delta>0\ \text{such that for any}$ $t\in U$ with $|t_0-t|<\delta,$ \[\max\{||\Phi_{t_0}-\Phi_{t}||_{\infty}, ||D\Phi_{t_0}-D\Phi_{t}||_{\infty}\}<\epsilon;\]
\item[(c)] There is a constant $0< \kappa <1$ such that for any $t\in U,$ any $j\in \mathbb{N}$, and any $i\in I^{(j)}$, 
\begin{align*}
|D\phi^{(j)}_{i, t}(x)|\ge\kappa
\end{align*}
for any $x\in X;$
\item[(d)]For any $\epsilon>0,$ there exists $\delta>0$ such that for any $t\in U$, any $j\in \mathbb{N},$ any $i\in I^{(j)},$ and any $x, y\in X$ with $|x-y|<\delta,$ \[\left||D\phi^{(j)}_{i, t}(x)|-|D\phi^{(j)}_{i, t}(y)|\right|<\epsilon.\]
\end{enumerate} Then the distortion continuity condition and the continuity condition hold.
\end{pro}
Note that the conditions (c) and (d) correspond to the conditions (Ka) and (Kb) in \cite[p.2010]{RU} respectively. 
\begin{proof}
We follow the proof of \cite[10.3 Theorem]{RU}. Fix $\eta>0$ and $t_0\in U.$ 
Then take $\epsilon>0$ such that \begin{align}\label{REM1}
\exp(2\kappa^{-1}\epsilon)\le \exp(\eta),\end{align} and take $\delta_1=\delta_1(\epsilon)$ such that the condition (d) holds. Furthermore, by the condition (b) there exists $\delta=\delta(\delta_1, t_0)$ such that for any $t\in U$ with $|t_0-t|<\delta,$ \begin{align*}\max\{||\Phi_{t_0}-\Phi_{t}||_{\infty}, ||D\Phi_{t_0}-D\Phi_{t}||_{\infty}\}<(1-\gamma)\min\{\epsilon, \delta_1\}.\end{align*}
Then by using the chain rule, the mean value theorem, the condition (c), and Claim \ref{CLAIM}, we have for any $t\in U$ with $|t_0-t|<\delta$, any $n, j\in \mathbb{N}$, any $\omega\in I_{n}^{n+j-1},$ and any $x\in X,$  
\begin{align*}
\left|\log\frac{|D\phi_{\omega, t_0}(x)|}{|D\phi_{\omega, t}(x)|}\right|&=\left|\sum_{k=n}^{n+j-1}\log\left|D\phi^{(k)}_{\omega_k, t_0}(\phi_{{\omega_{k+1}\cdots \omega_{n+j-1}}, t_0}(x))\right|-\log\left|D\phi^{(k)}_{\omega_k, t}(\phi_{{\omega_{k+1}\cdots \omega_{n+j-1}}, t}(x))\right|\right|\\
&\le \sum_{k=n}^{n+j-1}\left|\log\left|D\phi^{(k)}_{\omega_k, t_0}(\phi_{{\omega_{k+1}\cdots \omega_{n+j-1}}, t_0}(x))\right|-\log\left|D\phi^{(k)}_{\omega_k, t}(\phi_{{\omega_{k+1}\cdots \omega_{n+j-1}}, t}(x))\right|\right|\\
&= \sum_{k=n}^{n+j-1}\Biggl|\left(\log\left|D\phi^{(k)}_{\omega_k, t_0}(\phi_{{\omega_{k+1}\cdots \omega_{n+j-1}}, t_0}(x))\right|-\log\left|D\phi^{(k)}_{\omega_k, t}(\phi_{{\omega_{k+1}\cdots \omega_{n+j-1}}, t_0}(x))\right|\right)\\
&\hspace{30pt}+ \left(\log\left|D\phi^{(k)}_{\omega_k, t}(\phi_{{\omega_{k+1}\cdots \omega_{n+j-1}}, t_0}(x))\right|-\log\left|D\phi^{(k)}_{\omega_k, t}(\phi_{{\omega_{k+1}\cdots \omega_{n+j-1}}, t}(x))\right|\right)\Biggr|\\
&\le\sum_{k=n}^{n+j-1}\left|\log\left|D\phi^{(k)}_{\omega_k, t_0}(\phi_{{\omega_{k+1}\cdots \omega_{n+j-1}}, t_0}(x))\right|-\log\left|D\phi^{(k)}_{\omega_k, t}(\phi_{{\omega_{k+1}\cdots \omega_{n+j-1}}, t_0}(x))\right|\right|\\
&\hspace{30pt}+ \sum_{k=n}^{n+j-1}\left|\log\left|D\phi^{(k)}_{\omega_k, t}(\phi_{{\omega_{k+1}\cdots \omega_{n+j-1}}, t_0}(x))\right|-\log\left|D\phi^{(k)}_{\omega_k, t}(\phi_{{\omega_{k+1}\cdots \omega_{n+j-1}}, t}(x))\right|\right|\\
&\le\kappa^{-1}\sum_{k=n}^{n+j-1}\left|\left|D\phi^{(k)}_{\omega_k, t_0}(\phi_{{\omega_{k+1}\cdots \omega_{n+j-1}}, t_0}(x))\right|-\left|D\phi^{(k)}_{\omega_k, t}(\phi_{{\omega_{k+1}\cdots \omega_{n+j-1}}, t_0}(x))\right|\right|\\
&\hspace{30pt}+ \kappa^{-1}\sum_{k=n}^{n+j-1}\left|\left|D\phi^{(k)}_{\omega_k, t}(\phi_{{\omega_{k+1}\cdots \omega_{n+j-1}}, t_0}(x))\right|-\left|D\phi^{(k)}_{\omega_k, t}(\phi_{{\omega_{k+1}\cdots \omega_{n+j-1}}, t}(x))\right|\right|\\
&\le \kappa^{-1}j ||D\Phi_{t_0}-D\Phi_t||_{\infty}+\kappa^{-1}j\epsilon\\
&\le 2\kappa^{-1}j\epsilon. 
\end{align*}
Hence we have 
\[\exp(-2\kappa^{-1}j\epsilon)\le \frac{||D\phi_{\omega|_j, t_0}||_X}{||D\phi_{\omega|_j, t}||_X}\le \exp(2\kappa^{-1}j\epsilon).\] Then it follows from (\ref{REM1}) that the distortion continuity condition holds.
The continuity condition also holds by Proposition \ref{condtncifs2}.
\end{proof}
\section{Preliminaries and a proof of Main Theorem A}
In this section we give some lemmas for conformal maps on $\mathbb{R}^m$ and give the proof of Main Theorem A. Let $\{\Phi_t\}_{t\in U}=\left\{\left(\{\phi_{i, t}^{(j)}:X\rightarrow X\}_{i\in I^{(j)}}\right)_{j=1}^{\infty}\right\}_{t\in U}$ be a TNCIFS. 
\subsection{Lemma for conformal maps}
Let $n, k\in \mathbb{N}$ with $n\le k$. Below, we set $||D\phi_{\omega, t}||:=||D\phi_{\omega, t}||_X$ for any $\omega\in I_n^k$ and any $t\in U$. We set $I^{\ast}:=\cup_{n\ge 1}I^n$. This subsection is devoted to the proof of the following lemma.
\begin{lem}
\label{dist}There exists $L\ge 1$ such that for any $t\in U$, any $\omega\in I^{\ast}$ and any $x, y\in X,$
\begin{align}
\label{ine}
|\phi_{\omega, t}(x)-\phi_{\omega, t}(y)|\ge L^{-1}K(t)^{-2}||D\phi_{\omega, t}||\cdot|x-y|,
\end{align}
where $K(t)$ comes from the bounded distortion condition $(\ref{bd})$.
\end{lem}
We now prove Lemma \ref{dist} by imitating the argument in \cite[pages.73-74]{MoU2} as follows. We set $|X|=\sup_{x, y\in X}|x-y|(< \infty).$ For any set $A\subset \mathbb{R}^m$, we denote by $\partial A$ the boundary of $A$. Let $V$ be an open set with $V\supset X$ coming from the conformality condition in Definition \ref{tncifs}. We set $$r=\min\left\{|X|, \frac{\inf\{|x-y|\ :\ x\in X, y\in \partial V\}}{2}\right\}.$$ In order to prove Lemma \ref{dist}, we give the following lemma.
\begin{lem}
\label{dist1}
Let $t\in U$. For any $\omega\in I^{\ast}$ and $x\in X,$
$$\phi_{\omega, t}(B(x, r))\supset B(\phi_{\omega, t}(x), K(t)^{-1}||D\phi_{\omega, t}||r).$$
\end{lem}
\begin{proof}
Let $t\in U$. Fix $x\in X$. For any $\omega\in I^{\ast},$ we set $$R_{\omega}=\sup\{ u> 0\ :\ B(\phi_{\omega, t}(x), u)\subset \phi_{\omega, t}(B(x, r))\}.$$ Then 
\begin{align}\label{inter}\partial B(\phi_{\omega, t}(x), R_{\omega})\cap \partial \phi_{\omega, t}(B(x, r))\neq \emptyset.
\end{align}
Since $B(\phi_{\omega, t}(x), R_{\omega})\subset \phi_{\omega, t}(B(x, r))\subset \phi_{\omega, t}(V)$, by applying the mean value inequality to the map $\phi_{\omega, t}^{-1}$ restricted to the convex set $B(\phi_{\omega, t}(x), R_{\omega})$ and using the bounded distortion condition (\ref{bd}), we have 
\begin{align*}
\phi_{\omega, t}^{-1}(B(\phi_{\omega, t}(x), R_{\omega}))\subset B(x, ||D(\phi_{\omega, t}^{-1})||_{\phi_{\omega, t}(V)}R_{\omega})\subset B(x, K(t)||D\phi_{\omega, t}||^{-1} R_{\omega}).
\end{align*}
This implies 
\begin{align}
\label{hougan}
B(\phi_{\omega, t}(x), R_{\omega})\subset \phi_{\omega, t}(B(x, K(t)||D\phi_{\omega, t}||^{-1} R_{\omega})).
\end{align}
By (\ref{inter}) and (\ref{hougan}), we have $ K(t)||D\phi_{\omega, t}||^{-1} R_{\omega}\ge r.$ By the definition of $R_{\omega}$, we have $$\phi_{\omega, t}(B(x, r))\supset B(\phi_{\omega, t}(x), K(t)^{-1}||D\phi_{\omega, t}||r).$$
\end{proof}
We now give a proof of Lemma \ref{dist}.

\begin{proof}[Proof of Lemma $\ref{dist}$]
Let $t\in U$, $\omega\in I^{\ast}$, and $x, y\in X.$

(Case 1: $|x-y|\le K(t)^{-1}r$) By applying the mean value inequality to the map $\phi_{\omega, t}$ restricted to the convex set $B(x, K(t)^{-1}r)$ and using the bounded distortion condition (\ref{bd}), we have 
\begin{align*}
\phi_{\omega, t}(y)\in B(\phi_{\omega, t}(x), K(t)^{-1}||D\phi_{\omega, t}||r).
\end{align*}
Moreover, by Lemma \ref{dist1} we have 
\begin{align}
\label{hougan3}
\phi_{\omega, t}(y)\in B(\phi_{\omega, t}(x), K(t)^{-1}||D\phi_{\omega, t}||r)\subset \phi_{\omega, t}(B(x, r))\subset \phi_{\omega, t}(V).
\end{align}
By  (\ref{hougan3}) and applying the mean value inequality to the map $\phi_{\omega, t}^{-1}$ restricted to the convex set $B(\phi_{\omega, t}(x), K(t)^{-1}||D\phi_{\omega, t}||r)$, we have 
\begin{align*}
|x-y|&= |(\phi_{\omega, t})^{-1}(\phi_{\omega, t}(x))-(\phi_{\omega, t})^{-1}(\phi_{\omega, t}(y))|\\
&\le ||D(\phi_{\omega, t})^{-1}||_{\phi_{\omega, t}(V)}|\phi_{\omega, t}(x)-\phi_{\omega, t}(y)|.
\end{align*}
By using the bounded distortion condition (\ref{bd}), we have 
\begin{align}
\label{ine1}
|x-y|\le K(t)||D\phi_{\omega, t}||^{-1}\cdot|\phi_{\omega, t}(x)-\phi_{\omega, t}(y)|.
\end{align}
Hence we obtain (\ref{ine}).

(Case 2: $|x-y|> K(t)^{-1}r$) Since $\phi_{\omega, t}(y)\notin \phi_{\omega, t}(B(x, K(t)^{-1}r)),$ there exists $z\in \partial B(x,  K(t)^{-1}r)$ such that $\phi_{\omega, t}(z)$ belongs to the straight line path from $\phi_{\omega, t}(x)$ to $\phi_{\omega, t}(y)$. Hence 
\begin{align}
\label{ine2}
|\phi_{\omega, t}(x)-\phi_{\omega, t}(y)|\ge |\phi_{\omega, t}(x)-\phi_{\omega, t}(z)|.
\end{align}
Since $|x-z|= K(t)^{-1}r$, by (\ref{ine1}) we have 
\begin{align}
\label{ine3}
|\phi_{\omega, t}(x)-\phi_{\omega, t}(z)|\ge K(t)^{-1}||D\phi_{\omega, t}||\cdot |x-z|=K(t)^{-1}||D\phi_{\omega, t}||K(t)^{-1}r.
\end{align}
By (\ref{ine2}) and (\ref{ine3}) we have 
$$|\phi_{\omega, t}(x)-\phi_{\omega, t}(y)|\ge K(t)^{-1}||D\phi_{\omega, t}||K(t)^{-1}\frac{|x-y|r}{|x-y|}\ge \frac{r}{|X|}K(t)^{-2}||D\phi_{\omega, t}||\cdot{|x-y|}.$$If we set $L=|X|/r (\ge 1)$, then we obtain (\ref{ine}). Thus we have proved our lemma. 
\end{proof}
\subsection{Continuity of the map $t\mapsto s(t)$}
We consider the continuity of the map $U\ni t\mapsto s(t)\in [0, \infty],$ where $s(t)$ is the Bowen dimension of the system $\Phi_t$ (see (\ref{PF3}) for the definition). We give the following.
\begin{pro}
\label{CB1}
Let $\{\Phi_t\}_{t\in U}$ be a TNCIFS. Then the map $t\mapsto s(t)$ is continuous on $U$.
\end{pro}
\begin{proof}
Fix $t_0\in U.$ We now show that the map $t\mapsto s(t)$ is continuous at $t_0.$ By the distortion continuity (\ref{dc}), for any $\eta>0$, there exists $\delta=\delta(\eta, t_0)>0$ such that for any $t\in U$ with $|t-t_0|\le \delta$, for any $n\in \mathbb{N}$ and for any $\omega\in I^n,$
\begin{align*} 
\exp(-n \eta)\le \frac{||D\phi_{\omega, t_0}||}{||D\phi_{\omega, t}||}\le \exp(n \eta).
\end{align*}
Hence we have for any $n\in \mathbb{N}$, any $s\in [0, \infty)$ and any $t\in U$ with $|t-t_0|\le \delta,$
\begin{align*}
-s\eta+\frac{1}{n}\log Z_{n, t}(s)\le \frac{1}{n}\log Z_{n, t_0}(s)\le s\eta+\frac{1}{n}\log Z_{n, t}(s), 
\end{align*}
which implies that for any $s\in [0, \infty)$ and any $t\in U$ with $|t-t_0|\le \delta,$
\begin{align}
\label{CB2}
-s\eta+\underline{P}_t(s)\le \underline{P}_{t_0}(s)\le s\eta+\underline{P}_t(s)
\end{align}
(see (\ref{PF1}) and (\ref{PF2}) for the definitions of $Z_{n, t}(s)$ and $\underline{P}_t(s)$).
We divide the argument into three cases.

(Case 1: $s(t_0)=\infty$)
Take $s>0.$ Then by the definition of $s(t_0),$ we have $\underline{P}_{t_0}(s)>0.$ Set \[ \eta=\eta(s):=
\begin{cases}
 \frac{\underline{P}_{t_0}(s)}{2s}&(\underline{P}_{t_0}(s)< \infty)\\
 1&(\underline{P}_{t_0}(s)=\infty).
 \end{cases}
 \]
By (\ref{CB2}), we have for any $t\in U$ with $|t-t_0|\le \delta=\delta(\eta, t_0),$
\begin{align*}
\underline{P}_t(s)\ge \underline{P}_{t_0}(s)-s\eta>0.
\end{align*}
Then by the definition of $s(t),$ we have for any $t\in U$ with $|t-t_0|\le \delta=\delta(\eta, t_0),$ \[s(t)\ge s.\] Since $s$ is arbitrary, we have $\lim_{t\to t_0}s(t)=\infty=s(t_0).$

(Case 2: $s(t_0)\in (0, \infty)$)
Take $\epsilon>0$ with $s(t_0)-\epsilon>0.$ Then by the definition of $s(t_0),$ we have $\underline{P}_{t_0}(s(t_0)+\epsilon)<0$ and $\underline{P}_{t_0}(s(t_0)-\epsilon)>0.$ Set \[ \eta_1(\epsilon):=
\begin{cases}
 \frac{-\underline{P}_{t_0}(s(t_0)+\epsilon)}{2(s(t_0)+\epsilon)}&( \underline{P}_{t_0}(s(t_0)+\epsilon)>-\infty)\\
 1&(\underline{P}_{t_0}(s(t_0)+\epsilon)=-\infty),
 \end{cases}
 \]
  \[ \eta_2(\epsilon):=
\begin{cases}
 \frac{\underline{P}_{t_0}(s(t_0)-\epsilon)}{2(s(t_0)-\epsilon)}&( \underline{P}_{t_0}(s(t_0)-\epsilon)<\infty)\\
 1&(\underline{P}_{t_0}(s(t_0)-\epsilon)=\infty),
 \end{cases}
 \] and $\eta:=\min\{\eta_1(\epsilon), \eta_2(\epsilon)\}.$
By (\ref{CB2}), we have for any $t\in U$ with $|t-t_0|\le \delta=\delta(\eta, t_0),$
\begin{align*}
\underline{P}_{t}(s(t_0)+\epsilon)\le \underline{P}_{t_0}(s(t_0)+\epsilon)+(s(t_0)+\epsilon)\eta<0,
\end{align*}
and 
\begin{align*}
\underline{P}_{t}(s(t_0)-\epsilon)\ge \underline{P}_{t_0}(s(t_0)-\epsilon)-(s(t_0)-\epsilon)\eta>0.
\end{align*}
Then by the definition of $s(t),$ we have for any $t\in U$ with $|t-t_0|\le \delta=\delta(\eta, t_0),$ \begin{align}
\label{CB3}
 s(t_0)-\epsilon\le s(t)\le s(t_0)+\epsilon.
\end{align}
 Then by (\ref{CB3}), we have the map $t\mapsto s(t)$ is continuous at $t_0.$
 
(Case 3: $s(t_0)=0$) Take $\epsilon>0.$ By the definition of $s(t_0),$ we have $\underline{P}_{t_0}(\epsilon)<0.$ Set \[ \eta=\eta(\epsilon):=
\begin{cases}
 \frac{-\underline{P}_{t_0}(\epsilon)}{2\epsilon}&( \underline{P}_{t_0}(\epsilon)>-\infty)\\
 1&(\underline{P}_{t_0}(\epsilon)=-\infty).
 \end{cases}
 \]Then by the same argument as in the case 2, we have that there exists $\delta=\delta(\eta, t_0)>0$ such that for any $t\in U$ with $|t-t_0|\le \delta,$ \begin{align*}
 0 \le s(t)\le \epsilon.
\end{align*} Hence we have the map $t\mapsto s(t)$ is continuous at $t_0.$
\end{proof}
\subsection{Transversality argument}
For $\omega\in I^{\ast},$ let $|\omega|$ be the length of $\omega$. We prove the following two lemmas by imitating the proofs of Lemmas 3.2 and 3.3 in \cite{SSU}.
\begin{lem}
\label{lem1}
Let $\epsilon, a>0$ and $t_0\in U$. We set $\eta=\frac{-\epsilon \log \gamma}{4a+\epsilon}$ and take $\delta=\delta(\eta, t_0)$ coming from the distortion continuity $(\ref{dc})$ ascribed to $\eta$ and $t_0$, where $\gamma$ is the constant coming from the uniform contraction condition. Then for all $\omega\in I^{\ast}$ and $t\in U$ with $|t_0-t|\le \delta$, $||D\phi_{\omega, t_0}||^{a+\frac{\epsilon}{4}}\le ||D\phi_{\omega, t}||^{a}.$ 
\end{lem}
\begin{proof}
By the distortion continuity (\ref{dc}), we have 
\begin{align*}
||D\phi_{\omega, t_0}||^{a+\frac{\epsilon}{4}}&\le \exp\left(|\omega|\eta\left(a+\frac{\epsilon}{4}\right)\right)\cdot||D\phi_{\omega, t}||^{a+\frac{\epsilon}{4}}\\
&\le \exp\left(|\omega|\eta\left(a+\frac{\epsilon}{4}\right)\right)\gamma^{|\omega|\frac{\epsilon}{4}}||D\phi_{\omega, t}||^{a}\\
&(\mbox{by the uniform contraction condition})\\
&=\exp\left(|\omega|\left(\eta\left(a+\frac{\epsilon}{4}\right)+{\frac{\epsilon}{4}}\log\gamma\right)\right)\cdot||D\phi_{\omega, t}||^{a}.
\end{align*} 
\end{proof}
\begin{lem}
\label{lem2}
For any compact subset $G\subset U$ and any $\alpha$ with $0< \alpha< m$, there exists a sequence $\{\tilde{C}_n\}_{n=1}^{\infty}$ of positive constants such that $$\lim_{n\to \infty}\frac{\log \tilde{C}_n}{n}=0$$ and for any $\omega, \tau\in I_n^{\infty}$ with $\omega_n\neq \tau_n$,  
$$\int_G \frac{1}{|\pi_{n, t}(\omega)-\pi_{n, t}(\tau)|^{\alpha}}\ d\mathcal{L}_d(t)\le \tilde{C}_n.$$
\end{lem}
\begin{proof}Let $n\in \mathbb{N}$. By the transversality condition we have that
\begin{align*}
\int_G \frac{1}{|\pi_{n, t}(\omega)-\pi_{n, t}(\tau)|^{\alpha}}\ d\mathcal{L}_d(t)&=\int_0^{\infty} \mathcal{L}_d\left(\{t\in G\ :\ \frac{1}{|\pi_{n, t}(\omega)-\pi_{n, t}(\tau)|^{\alpha}}\ge x\}\right)\ dx \\
&=\int_0^{\infty} \mathcal{L}_d\left(\{t\in G\ :\ {|\pi_{n, t}(\omega)-\pi_{n, t}(\tau)|}\le \frac{1}{x^{1/\alpha}}\}\right)\ dx\\
&=\int_0^{|X|^{-\alpha}} \mdl(G)\ dx+\int_{|X|^{-\alpha}}^{\infty}C_n\frac{1}{x^{m/\alpha}}\ dx\\
&=|X|^{-\alpha}\mdl(G)+C_n\left[\frac{1}{1-m/\alpha}x^{1-m/\alpha}\right]_{|X|^{-\alpha}}^{\infty}\\
&=|X|^{-\alpha}\mdl(G)+C_n\frac{1}{m/\alpha-1}|X|^{m-\alpha}=:\tilde{C}_n.
\end{align*}
Since $\frac{1}{n}\log C_n\to 0$ as $n\to \infty$, we have $\frac{1}{n} \log \tilde{C}_n\to 0$ as $n\to \infty.$ 
\end{proof}
For any $\omega\in I^{\ast},$ we define a cylinder set $[\omega]$ as $[\omega]=\{\tau\in I^{\infty}\ :\ \tau_1=\omega_1,..., \tau_{|\omega|}=\omega_{|\omega|}\}$. We denote by $\delta_{\omega}$ the Dirac measure at $\omega\in I^{\infty}.$ We give a Gibbs-like measure by employing the proof of {\em Claim} in the proof of 3.2 Theorem in \cite{RU} and the argument in \cite[page. 232]{HZ}.
\begin{lem}[The existence of a Gibbs-like measure]
\label{lem3}
Let $t\in U$ and let $s\ge 0$. Then there exists a Borel probability measure $\mu_{t, s}$ on $I^{\infty}$ such that for any $\omega\in I^{\ast}$, 
\begin{align}
\label{gibbslike}
\mu_{t, s}([\omega])\le K(t)^{s}\frac{||D\phi_{\omega, t}||^s}{Z_{n, t}(s)},
\end{align}
where $K(t)$ is the constant coming from the bounded distortion $(\ref{bd})$ and $Z_{n, t}(s)=\sum_{\omega\in I^n}||D\phi_{\omega, t}||^s$.   
\end{lem}
\begin{proof}

Let $n\in \mathbb{N}$. For any $\omega\in I^n,$ take an element $\tau_{\omega}\in [\omega].$ 
For any $t\in U$, $s\ge 0$ and $n\in \mathbb{N}$, we define a Borel probability measure $\mu_{t, s, n}$ on $I^{\infty}$ as $$\mu_{t, s, n}=\frac{1}{Z_{n, t}(s)}\sum_{\omega\in I^n}||D\phi_{\omega, t}||^s\delta_{\tau_{\omega}}.$$ 
Then $$\mu_{t, s, n}([\omega])=\frac{||D\phi_{\omega, t}||^s}{Z_{n, t}(s)}$$ for any $\omega\in I^n$. 

If  $\omega\in I^n$, $\upsilon\in I_{n+1}^{n+j}$ and $\tau=\omega\upsilon\in I^{n+j}$, then by the bounded distortion (\ref{bd}), $||D\phi_{\omega, t}||\cdot||D\phi_{\upsilon, t}||\le K(t)||D\phi_{\tau, t}||.$ Hence for any $j\in \mathbb{N}$,
\begin{align}
\label{zetto}
Z_{n+j, t}(s)\ge \frac{1}{K(t)^{s}}Z_{n, t}(s)\sum_{\upsilon\in I_{n+1}^{n+j}}||D\phi_{\upsilon, t}||^s.
\end{align}

Thus we have that for any $j\in \mathbb{N}$ and for any $\omega\in I^n,$
\begin{align*}
\mu_{t, s, n+j}([\omega])&=\mu_{t, s, n+j}\left(\bigcup_{\upsilon\in I_{n+1}^{n+j}}[\omega\upsilon]\right)\\
&=\frac{\sum_{\upsilon\in I_{n+1}^{n+j}}||D\phi_{\omega\upsilon, t}||^s}{Z_{n+j, t}(s)}\\
&\le ||D\phi_{\omega, t}||^s\frac{\sum_{\upsilon\in I_{n+1}^{n+j}}||D\phi_{\upsilon, t}||^s}{Z_{n+j, t}(s)}\\
&\le K(t)^{s}\frac{||D\phi_{\omega, t}||^s}{Z_{n, t}(s)}\\
&(\mbox{by\ }(\ref{zetto})).
\end{align*}
Let $\mu_{t, s}$ be a weak$^{\ast}-$limit of a subsequence of $\{\mu_{t, s, j}\}_{j=1}^{\infty}$ in the space of Borel probability measures on $I^{\infty}$ (see e.g. \cite[Theorem 6.5]{W}). 
The above inequality implies 

$$\mu_{t, s}([\omega])\le K(t)^{s}\frac{||D\phi_{\omega, t}||^s}{Z_{n, t}(s)}.$$

\end{proof}
For any $n\in \mathbb{N},$ we define the map $\sigma^n: I^{\infty}\rightarrow I_{n+1}^{\infty}$ by $\sigma^n (\omega_1\omega_2\cdots)=\omega_{n+1}\omega_{n+2}\cdots.$ This is a continuous map with respect to the product topology. We give the following simple lemma.
\begin{lem}
\label{lemshift}
Let $t\in U.$ Then for any $n\in \mathbb{N}$ and $\omega\in I^{\infty},$
$$\pi_{1, t}(\omega)=\phi_{{\omega|_n}, t} (\pi_{{n+1}, t}(\sigma^n (\omega))).$$
\end{lem}
\begin{proof}
Let $t\in U.$ For any $n\in \mathbb{N}$ and $\omega\in I^{\infty},$ we have
\begin{align*}
\{\pi_{1, t}(\omega)\}&=\bigcap_{j=1}^{\infty}\phi_{\omega|_{j}, t}(X)\\
&=\bigcap_{j=n+1}^{\infty}\phi_{\omega|_{j}, t}(X)\\
&=\phi_{{\omega|_n}, t} \left(\bigcap_{j=1}^{\infty}\phi_{\sigma^n(\omega)|_{j}, t}(X)\right)\\
&=\phi_{{\omega|_n}, t} (\{\pi_{{n+1}, t}(\sigma^n (\omega)\}).
\end{align*}
\end{proof}
For any $\omega=\omega_1\omega_2\cdots, \tau=\tau_1\tau_2\cdots\in I^{\infty}$ with $\omega\neq \tau$ and $\omega_1=\tau_1$, we denote by $\omega\wedge \tau(\in I^{\ast})$ the largest common initial segment of $\omega$ and $\tau$. In order to prove Main Theorem A, we need the following which is the key lemma for the proof.
\begin{lem}
\label{lem4}
Let $\{\Phi_t\}_{t\in U}$ be a TNCIFS. Then for any $t_0\in U$ and any $\epsilon>0$, there exists $\delta=\delta(t_0, \epsilon)>0$ such that 
$$\dim_H(J_t)\ge \min\{m, s(t_0)\}-\frac{\epsilon}{2}$$
for $\mdl$- a.e. $t\in B(t_0, \delta)$.
\end{lem}
\begin{proof}
For any $t_0\in U,$ we set $s:=\min\{m, s(t_0)\}.$ We assume $s>0$, otherwise the statement holds. For any $0<\epsilon< 2s$, we set $$\eta=\frac{-\epsilon \log \gamma}{4\left(s-\frac{\epsilon}{2}\right)+\epsilon},$$ where $\gamma$ is the constant coming from the uniform contraction condition. Take $\delta=\delta(\eta, t_0)$ coming from the distortion continuity (\ref{dc}) ascribed to $\eta$ and $t_0$. 
By Lemma \ref{lem1}, for any $\omega\in I^{\ast}$ and $t\in B(t_0, \delta),$
\begin{align}
\label{lem42}
||D\phi_{\omega, t_0}||^{s-\frac{\epsilon}{4}}\le ||D\phi_{\omega, t}||^{s-\frac{\epsilon}{2}}.
\end{align}
Let $n\in \mathbb{N}$. For any $\rho\in I^n,$ we set 
\begin{align*}
&F:=\{(\omega, \tau)\in I^{\infty}\times I^{\infty}\ :\ \omega_1\neq\tau_1\},\\
&A_{\rho}:=\{(\omega, \tau)\in I^{\infty}\times I^{\infty}\ :\ \omega\wedge \tau=\rho\},\\
&H:=\{(\omega, \tau)\in I^{\infty}\times I^{\infty}\ :\ \omega=\tau\}.
\end{align*}
Then we have $I^{\infty}\times I^{\infty}=H \sqcup F\sqcup \bigsqcup_{n\ge 1} \bigsqcup_{\rho\in I^n} A_{\rho}$ (disjoint union). Let $\rho\in I^n.$
By Lemma \ref{lemshift} and Lemma \ref{dist}, there exists $L\ge 1$ such that for any $(\omega, \tau) \in A_{\rho}$ and $t\in U,$
\begin{align}
|\pi_{1, t}(\omega)-\pi_{1, t}(\tau)|&=|\phi_{\rho, t}(\pi_{{n+1}, t}(\sigma^n \omega))-\phi_{\rho, t}(\pi_{{n+1}, t}(\sigma^n \tau))| \notag \\
&\label{lem41}\ge L^{-1}K(t)^{-2}||D\phi_{\rho, t}||\cdot|\pi_{{n+1}, t}(\sigma^n \omega)-\pi_{{n+1}, t}(\sigma^n \tau)|.
\end{align}
Let $\mu=\mu_{t_0, s-\epsilon/4}$ be the Borel probability measure coming from Lemma \ref{lem3} ascribed to $t_0\in U$ and $s-\epsilon/4\ge 0.$ Since $\liminf_{n\to \infty}\frac{1}{n}\log Z_{n, t_0}(s-\frac{\epsilon}{4})>0,$ there exist $b>0$ and $n_0\in \mathbb{N}$ such that for all $n\ge n_0$, 
\begin{align}
\label{shisuu}
Z_{n, t_0}\left(s-\frac{\epsilon}{4}\right)> \exp(bn).
\end{align}
By (\ref{gibbslike}) and (\ref{shisuu}) we have for any $\omega\in I^{\infty},$ $\mu(\{\omega\})=0.$ Hence we obtain that 
\begin{align}
\label{diag}
(\mu\times \mu)(H)&=\int_{I^{\infty}}\mu\{\omega\in I^{\infty}\ :\ (\omega, \tau)\in H\}\ d\mu(\tau) \notag \\
&=\int_{I^{\infty}}\mu(\{\tau\})\ d\mu(\tau)=0.
\end{align} 
We set $\mu_2=\mu\times \mu$ and $$R(t):=\iint_{I^{\infty}\times I^{\infty}}\frac{1}{|\pi_{1, t}(\omega)-\pi_{1, t}(\tau)|^{s-\epsilon/2}}\ d\mu_2.$$
For simplicity, we use the convention that\[I^0=\{\emptyset\}, A_{\emptyset}=F, [\emptyset]=I^{\infty}, \phi_{\emptyset, t}={\rm id}_{\mathbb{R}^m}, \text{and}\  \sigma^0={\rm id}_{I^{\infty}}, \]where ${\rm id}_A$ is the identity map on the set $A$.
Then
\begin{align*}
\int_{B(t_0, \delta)}R(t)\ d\mdl(t)&=\sum_{n\ge 0}\sum_{\rho\in I^n}\iint_{A_\rho}\left(\int_{B(t_0, \delta)}\frac{1}{|\pi_{1, t}(\omega)-\pi_{1, t}(\tau)|^{s-\epsilon/2}}\ dt\right)\ d\mu_2(\omega, \tau)\\
&(\mbox{by Fubini's Theorem and}\ (\ref{diag}))\\
&\le \sum_{n\ge 0}\sum_{\rho\in I^n}\iint_{A_\rho}\left(\int_{B(t_0, \delta)}\frac{L^{s-\epsilon/2}K(t)^{2(s-\epsilon/2)}||D\phi_{\rho, t}||^{-s+\epsilon/2}}{|\pi_{{n+1},t}(\sigma^n \omega)-\pi_{{n+1}, t}(\sigma^n \tau)|^{s-\epsilon/2}}\ dt\right)\ d\mu_2(\omega, \tau)\\
&(\mbox{by\ } (\ref{lem41}))\\
&\le L^{s-\epsilon/2}\left(\sup_{t\in B(t_0, \delta)}K(t)^{2(s-\epsilon/2)}\right)\sum_{n\ge 0}\tilde{C}_{n+1}\sum_{\rho\in I^n}\iint_{A_\rho}||D\phi_{\rho, t_0}||^{-s+\epsilon/4}\ d\mu_2(\omega, \tau)\\
&(\mbox{by\ }(\ref{lem42})\ \mbox{and}\ \mbox{Lemma}\ \ref{lem2})\\
&\le  L^{s-\epsilon/2}\left(\sup_{t\in B(t_0, \delta)}K(t)^{2(s-\epsilon/2)}\right)\sum_{n\ge 0}\tilde{C}_{n+1}\sum_{\rho\in I^n}\iint_{A_\rho}\frac{K(t_0)^{s-\epsilon/4}}{\mu([\rho])Z_{n, t_0}(s-\epsilon/4)}\ d\mu_2(\omega, \tau)\\
&(\mbox{by}\ \mbox{Lemma}\ \ref{lem3})\\
&= {\rm Const.}\sum_{n\ge 0} \frac{\tilde{C}_{n+1}}{Z_{n, t_0}(s-\epsilon/4)}\sum_{\rho\in I^n}\frac{1}{\mu([\rho])}\iint_{A_\rho}\ d\mu_2(\omega, \tau)\\
&\left(\mbox{we set}\ {\rm Const.}= L^{s-\epsilon/2}\left(\sup_{t\in B(t_0, \delta)}K(t)^{2(s-\epsilon/2)}\right)K(t_0)^{s-\epsilon/4}\right)\\
&\le{\rm Const.}\sum_{n\ge 0} \frac{\tilde{C}_{n+1}}{Z_{n, t_0}(s-\epsilon/4)}\\
&(\mbox{since}\ \mu_2({A_\rho})\le \mu([\rho])^2).
\end{align*}
Since $\frac{1}{n} \log \tilde{C}_{n+1}\to 0$ as $n\to \infty$, it follows from (\ref{shisuu}) that
$$\int_{B(t_0, \delta)}R(t)\ \mdl(t)\le {\rm Const.}\sum_{n\ge 0} \frac{\tilde{C}_{n+1}}{Z_{n, t_0}(s-\epsilon/4)}< \infty.$$ 
Hence we have that for $\mdl$-a.e. $t\in B(t_0, \delta),$
$$R(t)=\int\int_{\mathbb{R}^m\times \mathbb{R}^m}\frac{1}{|x-y|^{s-\epsilon/2}}\ d\left(\pi_{1, t}(\mu)\times \pi_{1, t}(\mu)\right)< \infty,$$
where $\pi_{1, t}(\mu)$ is the push forward measure of $\mu$ by $\pi_{1, t}.$ Since $\pi_{1, t}(\mu)\left(J_{t}\right)=1,$ by \cite[Theorem 4.13 (a)]{Fal} we have $$\dim_H(J_t)\ge \min\{m, s(t_0)\}-\frac{\epsilon}{2}$$
for $\mdl$- a.e. $t\in B(t_0, \delta)$.
\end{proof}
In order to obtain more deeper results about $J_t$ corresponding to the parameter $t$ with $s(t)>m$, we need the following. Note that the set $\{t\in U\ : s(t)>m\}$ is open by the continuity of $s(t)$ (see Proposition \ref{CB1}).
\begin{lem}
\label{lem5}
Let $\{\Phi_t\}_{t\in U}$ be a TNCIFS. Then for any $t_0\in \{t\in U\ : s(t)>m\}$, there exists $\delta=\delta(t_0)>0$ such that 
$$\mathcal{L}_m (J_t)>0$$
for $\mdl$- a.e. $t\in B(t_0, \delta)$.
\end{lem}
\begin{proof}
Fix $t_0\in \{t\in U\ : s(t)>m\}.$ Take $\epsilon>0$ such that $s(t_0)>m(1+\epsilon/4)$ and set \[\eta=\frac{-\epsilon \log \gamma}{4+\epsilon}.\] 
Take $\delta=\delta(\eta, t_0)$ coming from the distortion continuity (\ref{dc}) ascribed to $\eta$ and $t_0$. 
By Lemma \ref{lem1}, for any $\omega\in I^{\ast}$ and $t\in B(t_0, \delta),$
\begin{align}
\label{LEM1}
||D\phi_{\omega, t_0}||^{1+\epsilon/4}\le ||D\phi_{\omega, t}||.
\end{align}
Let $\mu=\mu_{t_0, m(1+\epsilon/4)}$ be the Borel probability measure coming from Lemma \ref{lem3} ascribed to $t_0\in U$ and $m(1+\epsilon/4).$
It suffices to show that the push forward measure $\pi_{1, t}(\mu)$ of $\mu$ by $\pi_{1, t}$ is absolutely continuous with respect to $\mathcal{L}_m$ for $\mdl$- a.e. $t\in B(t_0, \delta)$. In order to do that we set \[\mathcal{S}:=\int_{B(t_0, \delta)}\int_{\mathbb{R}^m} \liminf_{r\to 0}\frac{\pi_{1, t}(\mu)(B(x, r))}{\mathcal{L}_m(B(x, r))}\ d\pi_{1, t}(\mu)(x)d\mathcal{L}_d(t).\]We remark that if $\mathcal{S}<\infty,$ then by \cite[2.12 Theorem]{M} we have $\pi_{1, t}(\mu)$ is absolutely continuous with respect to $\mathcal{L}_m$ for $\mdl$- a.e. $t\in B(t_0, \delta)$. We set $\mu_2=\mu\times \mu$ and denote the $m-$ dimensional Lebesgue measure of the unit ball by $b_m.$ We use the notations and convention introduced in the proof of Lemma \ref{lem4}. Then 
\begin{align*}
\mathcal{S}&\le \liminf_{r\to 0}b_m^{-1}r^{-m}\iint_{I^{\infty}\times I^{\infty}}\mathcal{L}_d(\{t\in B(t_0, \delta)\ :\ |\pi_{1, t}(\omega)-\pi_{1, t}(\tau)|<r\})\ d\mu_2(\omega, \tau)\\
&(\mbox{by Fatou's Lemma and Fubini's Theorem})\\
&=\liminf_{r\to 0}b_m^{-1}r^{-m}\sum_{n\ge 0}\sum_{\rho\in I^n}\iint_{A_\rho}\mathcal{L}_d(\{t\in B(t_0, \delta)\ :\ |\pi_{1, t}(\omega)-\pi_{1, t}(\tau)|<r\})\ d\mu_2(\omega, \tau)\\
&(\mbox{by (\ref{diag}}))\\
&\le \liminf_{r\to 0}b_m^{-1}r^{-m}\\
&\sum_{n\ge 0}\sum_{\rho\in I^n}\iint_{A_\rho}\mathcal{L}_d(\{t\in B(t_0, \delta)\ :\ |\pi_{n+1, t}(\sigma^{n}\omega)-\pi_{n+1, t}(\sigma^n\tau)|<rL K(t)^2||D\phi_{\rho, t}||^{-1}\})\ d\mu_2(\omega, \tau)\\
&(\mbox{by (\ref{lem41})})\\
&\le \liminf_{r\to 0}b_m^{-1}r^{-m}\sum_{n\ge 0}\sum_{\rho\in I^n}C_{n+1}r^m L^m \left(\sup_{t\in B(t_0, \delta)} K(t)\right)^{2m}||D\phi_{\rho, t_0}||^{-m(1+\epsilon/4)}\ \mu_2(A_{\rho})\\
&(\mbox{by (\ref{LEM1}) and the transversality condition})\\
&\le \liminf_{r\to 0}b_m^{-1}r^{-m}\sum_{n\ge 0}\sum_{\rho\in I^n}C_{n+1}r^m L^m \left(\sup_{t\in B(t_0, \delta)} K(t)\right)^{2m}\frac{K(t_0)^{m(1+\epsilon/4)}}{\mu([\rho])Z_{n, t_0}(m(1+\epsilon/4))} \mu_2(A_{\rho})\\
&(\mbox{by Lemma \ref{lem3}})\\
&\le {\rm Const.}\sum_{n\ge 0}\frac{C_{n+1} }{Z_{n, t_0}(m(1+\epsilon/4))}\ \left(\mbox{we set}\ {\rm Const.}=b_m^{-1}L^m \left(\sup_{t\in B(t_0, \delta)} K(t)\right)^{2m}K(t_0)^{m(1+\epsilon/4)}\right).
\end{align*}
Since $\frac{1}{n} \log {C}_{n+1}\to 0$ as $n\to \infty$, it follows from (\ref{shisuu}) that
$$\mathcal{S}\le {\rm Const.}\sum_{n\ge 0} \frac{{C}_{n+1}}{Z_{n, t_0}(m(1+\epsilon/4))}< \infty.$$ 
\end{proof}
\subsection{Proof of Main Theorem A}
The following is Main Theorem A.
\begin{thm}
\label{thm1}
Let $\{\Phi_t\}_{t\in U}$ be a TNCIFS. 
Then 
\begin{enumerate}
\item[(i)] \[\dim_H(J_t)= \min\{m, s(t)\}\ \text{for}\ \mdl-\text{a.e.}\ t\in U;\]

\item[(ii)] \[\mathcal{L}_m(J_t)>0\ \text{for}\ \mdl-\text{a.e.}\ t\in \{t\in U\ : s(t)>m\}.\]
\end{enumerate}
\end{thm}
\begin{proof}
By \cite[2.8 Lemma]{RU}, for any $t\in U$ we have $$\dim_H(J_t)\le \tilde{s}(t):=\min\{m, s(t)\}.$$ Hence it suffices to prove that $$\dim_H(J_t)\ge \tilde{s}(t)$$
for $\mdl$- a.e. $t\in U$. Suppose that this is not true. Then there exist $\epsilon>0$ and a Lebesgue density point $t_0\in U$ of the set
$$\{t\in U\ :\ \dim_H(J_t)< \tilde{s}(t)-\epsilon\}.$$ Then there exists $\delta_0>0$ such that for each $0<\delta< \delta_0$,
\begin{align}
\label{thm11}
\mathcal{L}_d(\{t\in B(t_0, \delta)\ : \dim_H(J_t)< \tilde{s}(t)-\epsilon\})> 0.
\end{align}
By the continuity of the function $\tilde{s}(t)$ (see Proposition \ref{CB1}), if $\delta$ is small enough then $\tilde{s}(t)< \tilde{s}(t_0)+\epsilon/2$ for all $t\in B(t_0, \delta).$ Thus for all $\delta$ sufficiently small we obtain from (\ref{thm11}) that 
\begin{align*}
\mathcal{L}_d(\{t\in B(t_0, \delta)\ : \dim_H(J_t)< \tilde{s}(t_0)-\epsilon/2\})> 0.
\end{align*}
This contradicts Lemma \ref{lem4} and completes the proof of the first part of our theorem. The second part follows from Lemma \ref{lem5} in a similar way.
\end{proof}

\section{Example}
In this section, we give a proof of Main Theorem B. We set $\mathbb{D}:=\{z\in \mathbb{C}\ :\ |z|< 1\}$. For any holomorphic function $f$ on $\mathbb{D}$, we denote by $f^{\prime}(z)$ the complex derivative of $f$ evaluated at $z\in \mathbb{D}$. We can prove the following slight variation of \cite[Lemma 5.2]{S2}. 
\begin{lem}
\label{cmpttrans1}
Let $\mathcal{H}$ be a compact subset of the space of holomorphic functions on $\mathbb{D}$ endowed with the compact open topology. We set $$\tilde{\mathcal{M}}_{\mathcal{H}}:=\{\la\in \mathbb{D}\ :\mbox{there exists}\  f\in \mathcal{H}\ {\rm such\ that}\ f(\la)=f^{\prime}(\la)=0\}.$$ Let $G$ be a compact  subset of $\mathbb{D}\backslash \tilde{\mathcal{M}}_{\mathcal{H}}.$ Then there exists $K=K(\mathcal{H}, G)>0$ such that for any $f\in \mathcal{H}$ and any $r>0,$
\begin{align}
\label{bound}
\m2l\left(\{\lambda\in G : |f(\lambda)|\le r\}\right)\le Kr^2.
\end{align}
\end{lem} 

We now give a family $\{\Phi_t\}_{t\in U}$ of parametrized systems such that $\{\Phi_t\}_{t\in U}$ is a TNCIFS but $\Phi_t$ does not satisfy the open set condition (\ref{osc}) for any $t\in U$. In order to do that, we set $$U:=\{t\in \mathbb{C}\ :\ |t|< 2\times 5^{-5/8},\ t\notin \mathbb{R}\}.$$ Note that $2\times 5^{-5/8}\approx 0.73143>1/\sqrt{2}.$ Let $t\in U$. For each $j\in \mathbb{N},$ we define $$\Phi^{(j)}_t=\{z\mapsto\phi^{(j)}_{1, t}(z), z \mapsto\phi^{(j)}_{2, t}(z)\}:=\left\{z \mapsto t z, z\mapsto tz+\frac{1}{j}\right\}.$$
\begin{pro}
\label{notosc}
For any $t\in U$, the system $\{\Phi^{(j)}_t\}_{j=1}^{\infty}$ does not satisfy the open set condition.
\end{pro}
\begin{proof}
Suppose that the system $\{\Phi^{(j)}_t\}_{j=1}^{\infty}$ satisfies the open set condition (\ref{osc}). Then there exists a compact subset $X\subset \mathbb{C}$ with ${\rm int}(X)\neq \emptyset$ such that $\phi^{(j)}_{1, t}({\rm int}(X))\cap \phi^{(j)}_{2, t}({\rm int}(X))=\emptyset.$ Hence there exist $x\in X$ and $r>0$ such that 
\begin{align*}
\phi^{(j)}_{1, t}(B(x, r))\cap \phi^{(j)}_{2, t}(B(x, r))=B(tx, |t|r)\cap B(tx+1/j, |t|r)=\emptyset.
\end{align*}
In particular, we have for all $j\in \mathbb{N},$ $$2|t|r< \frac{1}{j}.$$
This is a contradiction.
\end{proof}
We set $$X:=\left\{z\in \mathbb{C}\ :\ |z|\le \frac{1}{1-2\times 5^{-5/8}}\right\}.$$ Then we have that for any $t\in U$, for any $j\in \mathbb{N}$ and for any $i\in I^{(j)}:=\{1, 2\},$ $\phi^{(j)}_{i, t} (X)\subset X.$ We set $b^{(j)}_{1}=0$ and $b^{(j)}_{2}=1/j$ for each $j$. Let $n, j\in \mathbb{N}$. We give the following lemma.
\begin{lem}
\label{comp1}
Let $t\in U$. For any $\omega=\omega_{n}\cdots \omega_{n+j-1}\in I_{n}^{n+j-1}$ and any $z\in X$ we have
\begin{align*}
\phi_{\omega, t}(z)=\phi_{\omega_{n}, t}^{(n)}\circ\cdots\circ \phi_{\omega_{n+j-1}, t}^{(n+j-1)}(z)=t^{j}z+ \sum_{i=1}^{j}b^{(n+i-1)}_{\omega_{n+i-1}} t^{i-1} ,
\end{align*}
where $b^{(n+i-1)}_{\omega_{n+i-1}}\in \{0, \frac{1}{n+i-1}\}.$ In particular, for any $\omega=\omega_{n}\cdots \omega_{n+j-1}\cdots\in I_{n}^{\infty},$ $$\pi_{n, t}(\omega)=\sum_{i=1}^{\infty} b^{(n+i-1)}_{\omega_{n+i-1}} t^{i-1}.$$
\end{lem}
\begin{proof}
This can be shown by induction on $j$ and Remark \ref{rem1} (ii).
\end{proof}
We can show that the family $\{\Phi_t\}_{t\in U}$ of systems is a TNCIFS as follows. 
\begin{enumerate}
\item[1.] {\em Conformality} : Let $t\in U$. For any $j\in \mathbb{N}$ and any $i\in I^{(j)}$, $\phi^{(j)}_{i, t}(z)=t z+b^{(j)}_{i}$ is a similarity map on $\mathbb{C}$.
\item[2.] {\em Uniform Contraction} : We set $\gamma=2\times 5^{-5/8}.$ Then for any $\omega\in I_{n}^{n+j-1}$ and $z\in X$,
\begin{align*}
|D\phi_{\omega, t}(z)|= |t|^j\le \gamma^j
\end{align*}
by Lemma \ref{comp1}.
\item[3.] {\em Bounded distortion} : By Lemma \ref{comp1}, for any $\omega=\omega_{n}\cdots \omega_{n+j-1}\in I_{n}^{n+j-1}$ and $z\in \mathbb{C},$ $|D\phi_{\omega, t}(z)|=|t|^{j}.$ 
We define the Borel measurable locally bounded function $K: U\rightarrow [1, \infty)$ by $K(t)=1$. Then for any $\omega\in I_{n}^{n+j-1}$, 

\begin{align*}
|D\phi_{\omega, t}(z_1)|\le K(t)|D\phi_{\omega, t}(z_2)|
\end{align*}
for all $z_1,z_2\in \mathbb{C}$.
\item[4.] {\em Distortion continuity} : Fix $t_0\in U$. Since the map $t\mapsto \log|t|$ is continuous at $t_0\in U$, for any $\eta>0$ there exists $\delta=\delta(\eta, t_0)>0$ such that for any $t\in U$ with $|t_0-t|<\delta$, 
$$|\log|t_0|-\log|t||<\eta.$$
Hence we have 
$$|\log |t_0|^j/|t|^j|< j\eta.$$
Thus we have that for any $\omega\in I_{n}^{n+j-1},$
\begin{align*}
\exp(-j\eta)< \frac{||D\phi_{\omega, t_0}||}{||D\phi_{\omega, t}||}=\exp(\log |t_0|^j/|t|^j)<\exp(j\eta).
\end{align*}
\item[5.] {\em Continuity} : By Lemma \ref{comp1}, we have for any $t\in U$ and any $\omega\in I_{n}^{\infty},$ 
$$\pi_{n, t}(\omega)=\sum_{i=1}^{\infty} b^{(n+i-1)}_{\omega_{n+i-1}} t^{i-1}.$$ Hence the map $(t, \omega)\mapsto \pi_{n, t}(\omega)$ is continuous on $U\times I_n^{\infty}$.
\item[6.] {\em Transversality condition} : 
We introduce a set $\mathcal{G}$ of holomorphic functions on $\mathbb{D}$ and the set $\tilde{\mathcal{M}}_{\mathcal{G}}$ of double zeros in $\mathbb{D}$ for functions belonging to $\mathcal{G}$.
\begin{align*}
&\mathcal{G}:=\left\{f(t)=\pm 1+\sum_{j=1}^{\infty}a_j t^j : a_j\in [-1, 1]\right\},\\
&\tilde{\mathcal{M}}_{\mathcal{G}}:=\{t\in \mathbb{D}\ :\mbox{there exists}\  f\in \mathcal{G}\ {\rm such\ that}\ f(t)=f^{\prime}(t)=0\}.
\end{align*}
Note that $\mathcal{G}$ is a compact subset of the space of holomorphic functions on $\mathbb{D}$ endowed with the compact open topology. Let $n\in \mathbb{N}$.
Then we have for any $t\in U$ and any $\omega, \tau\in I_n^{\infty}$ with $\omega_n\neq \tau_n,$
\begin{align*}
\pi_{n, t}(\omega)-\pi_{n, t}(\tau)&=\sum_{i=1}^{\infty}b^{(n+i-1)}_{\omega_{n+i-1}} t^{i-1}-\sum_{i=1}^{\infty}b^{(n+i-1)}_{\tau_{n+i-1}} t^{i-1}\\
&=b^{(n)}_{\omega_n}-b^{(n)}_{\tau_n}+\sum_{i=2}^{\infty}\left( b^{(n+i-1)}_{\omega_{n+i-1}}-b^{(n+i-1)}_{\tau_{n+i-1}}\right) t^{i-1}\\
&=\frac{1}{n}\left(\pm 1+\sum_{i=2}^{\infty}n\left( b^{(n+i-1)}_{\omega_{n+i-1}}-b^{(n+i-1)}_{\tau_{n+i-1}}\right) t^{i-1}\right).
\end{align*}
Then the function $t\mapsto \pm 1+\sum_{i=2}^{\infty}n( b^{(n+i-1)}_{\omega_{n+i-1}}-b^{(n+i-1)}_{\tau_{n+i-1}}) t^{i-1}$ is a holomorphic function which belongs to $\mathcal{G}.$ Let $G\subset \mathbb{D}\backslash \tilde{\mathcal{M}}_{\mathcal{G}}$ be  a compact subset. By Lemma \ref{cmpttrans1}, there exists $K=K(\mathcal{G}, G)>0$ such that for any $\omega,\tau\in I_n^{\infty}$ with $\omega_n\neq \tau_n$ and any $r>0$,   

\begin{align*}
&\mathcal{L}_2(\{t\in G\ :\ |\pi_{n, t}(\omega)-\pi_{n, t}(\tau)|\le r\})\\&=\mathcal{L}_2(\{t\in G\ :\ |\pm 1+\sum_{i=2}^{\infty}n( b^{(n+i-1)}_{\omega_{n+i-1}}-b^{(n+i-1)}_{\tau_{n+i-1}}) t^{i-1}|\le nr\})\\
&\le K(nr)^2.
\end{align*}
If we set $C_n:=Kn^2$ for any $n\in \mathbb{N},$ we have
$$\mathcal{L}_2(\{t\in G\ :\ |\pi_{n, t}(\omega)-\pi_{n, t}(\tau)|\le r\})\le C_n r^2 $$
and
$$\frac{1}{n}\log C_n=\frac{1}{n}\log {K}+\frac{2}{n}\log n \to 0$$ as $n\to \infty$.

Finally, we use the following theorem.
\begin{thm}\cite[Proposition 2.7]{SX}
\label{bbbp}
A power series of the form $1+\sum_{j=1}^{\infty}a_j z^j,$ with $a_j\in [-1, 1],$ cannot have a non-real double zero of modulus less than $2\times 5^{-5/8}$. 
\end{thm}
By using above theorem, we have that $U=\{t\in \mathbb{C}\ :\ |t|< 2\times 5^{-5/8},\ t\notin \mathbb{R}\}\subset \mathbb{D}\backslash \tilde{\mathcal{M}}_{\mathcal{G}}.$ Hence the family $\{\Phi_t\}_{t\in U}$ satisfies the transversality condition.

By the above arguments, we obtain the following.
\begin{pro}
\label{extncifs}
The family  $\{\Phi_t\}_{t\in U}$ of parametrized systems is a TNCIFS.
\end{pro}
\subsection*{Acknowledgement.} 
The author would like to express his gratitude to an anonymous referee for the careful reading, valuable comments and the improvements of Main Theorem A. The author also would like to express his gratitude to Professor Hiroki Sumi for his valuable comments. 


\end{enumerate}

\end{document}